\documentclass[11pt]{amsart}
\usepackage{amssymb}

\theoremstyle{plain}
\newtheorem{theorem}{Theorem}[section]
\newtheorem{corollary}[theorem]{Corollary}
\newtheorem{lemma}[theorem]{Lemma}
\newtheorem{proposition}[theorem]{Proposition}

\theoremstyle{remark}

\theoremstyle{definition}

\begin{document}

\title[Unitary representations of $L^0(\mu,
{\mathbb T})$ and $C(M, {\mathbb T})$]{Unitary representations
of the groups of\\ measurable and continuous functions with\\ values in the circle}

\author{S{\l}awomir Solecki}

\address{Department of Mathematics\\
1409 W. Green St.\\
University of Illinois\\
Urbana, IL 61801, USA}

\email{ssolecki@math.uiuc.edu}



\thanks{Research supported by NSF grant DMS--1001623.}

\begin{abstract}
We give a classification of unitary representations of certain
Polish, not necessarily locally compact, groups: the groups of all
measurable functions with values in the circle and the groups of all continuous functions
on compact, second countable, zero-dimensional spaces
with values in the circle.
In the proofs of our classification results, certain structure theorems and
factorization theorems for linear operators are used.
\end{abstract}

\maketitle

\section{Introduction}

We study unitary representations of the groups of measurable and continuous functions with values in the circle.
A description of unitary representations of such groups is of interest especially in view of recent considerable
activity around topological and measurable dynamics of these groups; see the remarks below. The reader may consult \cite{Pes} for
background information on dynamics of Polish non-locally compact groups. (Recall that a topological group is
{\em Polish} if its topology is metrizable by a complete separable metric.)

For a Borel probability measure $\mu$ on a standard Borel space
and a topological group $H$, let $L^0(\mu,H)$ be the
topological group of all $\mu$-equivalence classes of
$\mu$-measurable functions with values in $H$. The multiplication
on $L^0(\mu, H)$ is implemented pointwise and the topology is the
convergence in measure topology. Groups of this form were perhaps first
systematically considered in \cite{HM} to provide an
embedding of each topological group into a connected group. Recently, the more particular Polish
groups $L^0(\mu, {\mathbb T})$ generated substantial interest in
the context of extreme amenability \cite{Gl}, measure preserving
group actions \cite{GW}, representation properties of
Polish groups \cite{Pes2}, and generic properties of monothetic
subgroups of certain large groups \cite{MT}, \cite{S}.

We also consider groups $C(M,{\mathbb T})$ of all continuous
functions from $M$ to the circle group ${\mathbb T}$,
where $M$ is a compact, second countable space. We take
$C(M, {\mathbb T})$ with the pointwise multiplication and with the
topology of uniform convergence. This arrangement makes $C(M,
{\mathbb T})$ into a Polish group. Recently, measure preserving actions of $C(M, {\mathbb T})$
were studied in \cite{MS}.
Observe, that groups described above are usually not locally compact: $L^0(\mu, {\mathbb T})$ is not locally compact when $\mu$
is not purely atomic and neither is $C(M, {\mathbb T})$ when $M$ is infinite.

The goal of the present paper 
is to give classifications of strongly continuous unitary representations
of the Polish groups $L^0(\mu, {\mathbb T})$ and $C(M, {\mathbb T})$ when $M$ is zero-dimensional; see Theorems~\ref{T:LLmain} and \ref{T:CCmain}.
(Recall that a space is {\em zero-dimensional} if it has a topological basis consisting of sets
that are both closed and open; a typical example is the Cantor set.)
Roughly speaking, the classification results say that the groups
$L^0(\mu, {\mathbb T})$ and $C(M, {\mathbb T})$ behave like
infinite dimensional tori. Despite the fact that these groups do
not have irreducible unitary representations of dimension greater
than $1$ (even of dimension greater than $0$ in the case of
$L^0(\mu, {\mathbb T})$ for $\mu$ without atoms), their unitary
representations can be constructed as {\em countable} direct sums of
simple building blocks with the blocks being {\em uniquely determined}
by the representation; see Subsections~\ref{Su:des2} and
\ref{Su:des} for a description of these blocks.

In the proof of the classification result for $L^0(\mu, {\mathbb T})$, Theorem~\ref{T:LLmain}, besides the
spectral theorem, an important role is played by
Kwapie{\'n}'s structure theorem for linear operators between
linear $L^0$ spaces. One consequence of our result,
Corollary~\ref{C:measl}, is the existence, for a given continuous
unitary representation of $L^0(\mu, {\mathbb T})$, of a unique up
to measure equivalence, smallest with respect to the relation of
absolute continuity, finite Borel measure $\nu$ that is absolutely continuous with respect to $\mu$
and is such
that the representation is the composition of the natural
homomorphism $L^0(\mu, {\mathbb T})\to L^0(\nu,
{\mathbb T})$ and a continuous unitary representation of $L^0(\nu,
{\mathbb T})$. The proof of the classification result for $C(M, {\mathbb T})$, Theorem~\ref{T:CCmain}, uses
the classification of unitary representations of $L^0(\mu, {\mathbb T})$ and certain
factorization theorems for linear operators. As a consequence of
Theorem~\ref{T:CCmain}, we get that given a continuous unitary
representation of $C(M, {\mathbb T})$ there exists a unique up
to measure equivalence, smallest with respect to the relation of
absolute continuity, finite Borel measure $\nu$ on $M$ such
that the representation is the composition of the natural
homomorphism $C(M, {\mathbb T})\to L^0(\nu, {\mathbb T})$ and a
continuous unitary representation of $L^0(\nu, {\mathbb T})$. This
result is included in Corollary~\ref{C:measc}.

\subsection*{Notation and conventions} By $\mathbb R$, $\mathbb C$,
and $\mathbb T$ we denote the real numbers, the complex numbers,
and the multiplicative group $\{ z\in {\mathbb C}\colon |z|=1\}$,
respectively. The following spaces will be involved in our
considerations: $C(M, G)$, $L^p(\mu, G)$, where $M$ is a compact
metrizable space, $p=0,2$, $\mu$ is a Borel probability measure on
a standard Borel space, and $G={\mathbb T}, {\mathbb R}, {\mathbb
C}^k$, for $k\in {\mathbb N}$. When $G={\mathbb T}$ these spaces will be regarded as groups,
when $G={\mathbb R}$ or $G={\mathbb C}^k$ they will be regarded as
linear spaces over $\mathbb R$ or $\mathbb C$, respectively.
Unless otherwise stated (and this option will be exercised),
$C(M,G)$ is equipped with the uniform convergence topology,
$L^0(\mu,G)$ with the convergence in measure topology, and
$L^2(\mu, G)$ with the $L^2$ topology. Note however that on
$L^0(\mu, {\mathbb T})$ the convergence in measure and the $L^2$
topology coincide.
The unitary group of a complex Hilbert space $H$ will be denoted by ${\mathcal U}(H)$ and it
will always be considered with its strong operator topology.

\smallskip

I thank Ilijas Farah for our discussions of extreme
amenability in 2004, Marius Junge for pointing out to me Pisier's book
\cite{P}, and Vladimir Pestov for valuable comments.

\section{Unitary representations of $L^0(\mu, {\mathbb T})$}

Fix a Borel probability measure $\mu$ on a
standard Borel space $X$. We are interested in describing all
continuous unitary representations of $L^0(\mu, {\mathbb T})$.

Fix a linear order $<_X$ on $X$ of which we will assume that the order topology
it generates is compact, second countable and the Borel sets with respect to this topology
coincide with the Borel sets on $X$. For example, we can fix a Borel isomorphism from
$X$ to a closed subset of $[0,1]$ and use it to pull back the standard linear order on $[0,1]$.
The order is important for the uniqueness part of Theorem~\ref{T:LLmain}.

\subsection{Description of representations}\label{Su:des2}

Assume that we are given a sequence $\kappa = (k_1, \dots , k_n)$
of elements of ${\mathbb Z}\setminus \{ 0\}$ with
\begin{equation}\label{E:gro1}
k_1\leq k_2\leq \cdots \leq k_n.
\end{equation}
Assume we have a finite Borel measure $\lambda$ on $X^n$ whose marginal measures
are absolutely continuous with respect to $\mu$, that is,
for $i\leq n$
\begin{equation}\label{E:zer2}
(\pi_i)_*(\lambda)\ll \mu, \tag{${\rm A}1$}
\end{equation}
where, for $i\leq n$, $\pi_i\colon X^n\to X$ is the projection on
the $i$-th coordinate. With this set of data we associate the
following representation of $L^0(\mu,{\mathbb T})$ on
$L^2(\lambda, {\mathbb C})$:
\[
L^0(\mu, {\mathbb T})\ni f\to U_f\in {\mathcal U}(L^2(\lambda, {\mathbb C})),
\]
where for $h\in L^2(\lambda, {\mathbb C})$
\[
U_f(h) = (\prod_{i\leq n} (f\circ \pi_i)^{k_i}) h.
\]
Thus, the bounded function $\prod_{i\leq n} (f\circ \pi_i)^{k_i}$
acts on $h\in L^2(\lambda, {\mathbb C})$ by multiplication.
Condition \eqref{E:zer2} ensures that the representation is well
defined. We denote this representation by $\sigma(\kappa,
\lambda)$. We do allow $\lambda$ to be the zero measure, in which
case $\sigma(\kappa, \lambda)$ is the trivial representation.

We will consider the following two additional conditions on the
finite measure $\lambda$ as above: for $1\leq i<j\leq n$
\begin{equation}\label{E:zer1}
\lambda(\{ (x_1, \dots , x_n)\in X^n\colon x_i=x_j\}) =0,\tag{${\rm A}2$}
\end{equation}
and for $1\leq i<j\leq n$ with $k_i=k_j$,
\begin{equation}\label{E:zer3}
\lambda(\{ (x_1, \dots , x_n)\in X^n\colon x_j <_X x_i\}) =0, \tag{${\rm A}3$}
\end{equation}
Conditions \eqref{E:zer1} and \eqref{E:zer3} are needed for the
uniqueness part of the theorem below.

\subsection{Statement of the main result}

Let $S$ be the set of all sequences $\kappa = (k_1, \dots , k_n)$
of elements of ${\mathbb Z}\setminus \{ 0\}$ with property
\eqref{E:gro1}. We say that the natural number $n$ in $\kappa = (k_1, \dots , k_n)$ is the length of $\kappa$ and
denote it by $|\kappa|$.

\begin{theorem}\label{T:LLmain}
Let $\phi$ be a continuous unitary representation of
$L^0(\mu,{\mathbb T})$ on a separable complex Hilbert space $H$.
Consider $H_0$, the orthogonal complement of
\[
\{ v\in H\colon \forall f\in L^0(\mu, {\mathbb T})\; \phi(f)(v) =
v\}.
\]

\noindent {\em Existence}: For $\kappa\in S$ and $i\in {\mathbb
N}$, there exist finite Borel measures $\lambda_\kappa^i$ on
$X^{|\kappa|}$ with properties \eqref{E:zer2},
\eqref{E:zer1}, \eqref{E:zer3}, and with
\begin{equation}\label{E:absas}
\lambda_\kappa^j\ll \lambda_\kappa^i, \hbox{ for }i<j,\tag{${\rm A}4$}
\end{equation}
such that the representation $\phi$ restricted to $H_0$ is
the direct sum of the representations
$\sigma(\kappa, \lambda_\kappa^i)$ with $\kappa\in S$ and $i\in
{\mathbb N}$.

\noindent {\em Uniqueness}: If the restriction of $\phi$ to $H_0$
is presented as the direct sum of $\sigma(\kappa,
(\lambda')^i_{\kappa})$, for $\kappa\in S$ and $i\in {\mathbb N}$,
with \eqref{E:zer2}, \eqref{E:zer1}, \eqref{E:zer3}, and
\eqref{E:absas}, then, for each $i$ and $\kappa$,
$\lambda^i_\kappa$ and $(\lambda')^i_\kappa$ are absolutely
continuous with respect to each other.
\end{theorem}

We point out the following corollary. Recall first that if $\mu$
and $\nu$ are finite Borel measures on a standard Borel space $X$
with $\nu\ll \mu$, then there is the natural surjective
homomorphism $L^0(\mu, {\mathbb T}) \to L^0(\nu, {\mathbb T})$;
simply note that the $\nu$-equivalence class of a Borel function
from $X$ to $\mathbb T$ contains its $\mu$-equivalence class.

\begin{corollary}\label{C:measl}
Let $\mu$ be a Borel probability measure on a standard Borel space
$X$. Let $\phi$ be a continuous unitary representation of
$L^0(\mu, {\mathbb T})$. There exists a finite Borel measure $\nu$
on $X$ such that
\begin{enumerate}
\item[(i)] $\nu\ll\mu$ and $\phi$ is the composition of the
natural homomorphism from $L^0(\mu, {\mathbb T})$ to $L^0(\nu,
{\mathbb T})$ and a continuous unitary representation of $L^0(\nu,
{\mathbb T})$;

\item[(ii)] if $\nu'$ is a finite Borel measure on $X$ with
$\nu'\ll \mu$ and such that $\phi$ is the composition of the
natural homomorphism from $L^0(\mu, {\mathbb T})$ to $L^0(\nu',
{\mathbb T})$ and a continuous unitary representation of
$L^0(\nu', {\mathbb T})$, then $\nu\ll \nu'$.
\end{enumerate}
\end{corollary}

\begin{proof} We keep the notation from Theorem~\ref{T:LLmain}. Let
$\lambda^j_\kappa$, $\kappa\in S,\, i\in {\mathbb N}$, give a
presentation of $\phi$ restricted to $H_0$. For each $\kappa\in S$
consider the $|\kappa|$ many projections $X^{|\kappa|}\to X$ and
form push-forward measures on $X$ using these projections and
measures $\lambda^j_\kappa$ as $i$ varies over $\mathbb N$. As
$\kappa$ varies over $S$, this procedure gives countably many
finite Borel measures on $X$. Form their weighted sum with
positive coefficients to obtain a finite Borel measure $\nu$ on
$X$. This measure clearly fulfils (i).

To see (ii), fix $\nu'$ as in the assumption. Consider the
continuous unitary representation of $L^0(\nu', {\mathbb T})$ as
in this assumption. The presentation of this representation on
$H_0$ as in Theorem~\ref{T:LLmain} is also a presentation of
$\phi$. By the uniqueness part of Theorem~\ref{T:LLmain}, the
corresponding measures in these two presentations are absolutely
continuous with respect to each other. It follows that the
push-forward of each $\lambda^j_\kappa$ by each projection as in
the previous paragraph is absolutely continuous with respect to
$\nu'$ and (ii) now follows by our definition of $\nu$.
\end{proof}

\subsection{Background results}\label{Su:bac}

We state here three results that will be needed in the sequel. The
first one is a factorization result, the second is the form of the
spectral theorem needed in this paper, and the third is a
structure theorem for certain linear operators.

The first result is implicit in the proof of \cite[Theorem 5]{HC};
see page 208 of \cite{HC}.

\begin{proposition}\label{P1}
Let $\nu$ be a probability measure on a standard Borel space. Let
$G$ be a real topological vector space that is separable and
completely metrizable. Then for each continuous homomorphism
$\phi\colon G\to L^0(\nu, {\mathbb T})$ there exists a continuous
linear operator $\theta\colon G\to L^0(\nu, {\mathbb R})$ such
that
\[
\phi = \exp(2\pi i \theta).
\]
\end{proposition}

The following result is the version of the spectral theorem that
we need later on; see \cite[Proposition 4.7.13]{Pe}.

\begin{proposition}\label{P:sp}
Let $\mathcal T$ be a family of commuting unitary operators on a
separable complex Hilbert space $H$. Then there exist a Borel
probability measure $\nu$ on a standard Borel space and a
surjective isometric operator $U\colon H \to L^2(\nu, {\mathbb
C})$ such that for each $T\in {\mathcal T}$, $UTU^{-1}$ is the
multiplication operator on $L^2(\nu, {\mathbb C})$ by an element
of $L^0(\nu, {\mathbb T})$.
\end{proposition}

The following theorem of Kwapie{\'n} \cite{K} (see also
\cite[Theorem 8.4]{KPR}) describing the structure of continuous
linear operators from $L^0$ to $L^0$ is important in our
argument. Recall that a measurable function $\sigma$ between two
standard Borel spaces with Borel probability measures is called
{\em non-singular} if preimages under $\sigma$ of measure zero
sets are of measure zero.

\begin{proposition}\label{P:Kwap}
Let $\mu$, $\nu$ be Borel probability measures on standard Borel
spaces $X$ and $Y$, respectively. Let $T\colon L^0(\mu, {\mathbb
R}) \to L^0(\nu, {\mathbb R})$ be a continuous linear function.
Then there exist non-singular maps $\sigma_n\colon Y\to X$ and
$g_n\in L^0(\nu, {\mathbb R})$, with $g_n(x)= 0$ holding for
$\nu$-almost all $x\in Y$  and for all but finitely many $n$, such that for
$f\in L^0(\mu, {\mathbb R})$ we have
\[
T(f) = \sum_{n\in {\mathbb N}} g_n\cdot (f\circ \sigma_n).
\]
\end{proposition}

\subsection{Proof of Theorem~\ref{T:LLmain}}\label{Su:prL}

In this proof, when we say ``representation" we mean ``strongly continuous
unitary representation." In the proof, we write $L^2(\nu)$ for $L^2(\nu, {\mathbb C})$.

\smallskip

\noindent {\bf Proof of existence of the presentation.} Assume
we have a representation
\[
\phi\colon L^0(\mu,{\mathbb T}) \to {\mathcal U}(H)
\]
on a separable complex Hilbert space $H$. We divide the proof into
three steps.

{\em Step 1.} We show that for each closed non-trivial subspace $H_1$
included in $H_0$ (as defined in the theorem) and invariant under
the representation, there exists a closed non-trivial subspace
$H'$ of $H_1$ invariant under the representation such that the
representation restricted to $H'$ is of the form
$\sigma(\kappa, \lambda)$ for some $\kappa\in S$ and some
$\lambda$ fulfilling \eqref{E:zer2} and \eqref{E:zer1}. (Condition
\eqref{E:zer3} will be dealt with in Step 2.)

We note that $H_0$ is invariant under the representation as is its
orthogonal complement. Therefore, for simplicity of notation, we
can assume that $H_0=H$ and, in fact, for the
same reason, we assume that $H_1=H$. Since $L^0(\mu,{\mathbb
T})$ is abelian, it follows from the spectral theorem, Proposition~\ref{P:sp},
that there exists a Borel probability measure $\nu$ on a standard Borel
space $Y$ such that the representation $\phi$ is of the form
\[
L^0(\mu, {\mathbb T})\ni f\to N_f\in {\mathcal U}(L^2(\nu))
\]
where
\[
N_f(h) = \psi(f)\cdot h
\]
for some continuous homomorphism
\[
\psi\colon L^0(\mu, {\mathbb T})\to L^0(\nu, {\mathbb T}).
\]
By precomposing $\psi$ with the exponential homomorphism
\[
L^0(\mu, {\mathbb R})\to L^0(\mu, {\mathbb T});\; f\to \exp(2\pi i
f),
\]
we obtain a continuous homomorphism $\psi' \colon L^0(\mu,{\mathbb
R})\to L^0(\nu, {\mathbb T})$. By Proposition~\ref{P1}, there
exists a continuous linear operator
\[
\theta\colon L^0(\mu,{\mathbb R})\to L^0(\nu, {\mathbb R})
\]
such that
\[
\psi' = \exp(2\pi i\theta).
\]

Using Kwapie{\'n}'s theorem, Proposition~\ref{P:Kwap}, applied
to the operator $\theta$, we find $\nu$-measurable functions $g_n\colon Y\to
{\mathbb R}$ and non-singular functions $\sigma_n\colon Y\to X$,
$n\in {\mathbb N}$, such that for $\nu$-almost all $x\in Y$ only
finitely many $g_n(x)$ are non-zero and
\[
\theta(f) = \sum_n g_n \cdot(f\circ\sigma_n).
\]
It follows that for each $f\in L^0(\mu, {\mathbb R})$ we have
\[
\psi(\exp(2\pi i f)) = \exp(2\pi i \sum_n g_n\cdot
(f\circ\sigma_n)).
\]
The above equality implies that for $f\in L^0(\mu, {\mathbb R})$
\begin{equation}\label{E:inva}
f \hbox{ has integer values}\Longrightarrow
\sum_n g_n \cdot (f\circ\sigma_n) \hbox{ has integer values.}
\end{equation}

Now we partition $Y$ up to a $\nu$-measure zero set into countably
many $\nu$-measurable sets $A$ for which there is a finite set
$D\subseteq {\mathbb N}$ depending on $A$ such that for every
$y\in A$
\[
D=\{ n \colon g_n(y)\not= 0 \}
\]
and for all $y,y'\in A$ and $m,n\in D$
\begin{equation}\label{E:gzy}
\sigma_m(y) = \sigma_n(y) \;\Longleftrightarrow\; \sigma_m(y') =
\sigma_n(y').
\end{equation}
This is done as follows. Each finite $D\subseteq {\mathbb N}$
determines a $\nu$-measurable set
\[
A_D = \{ y\in Y\colon \{ n\colon g_n(y)\not=0\} =D\}.
\]
Each $y\in A_D$ induces an equivalence relation $E_y$ on $D$ by
the formula
\[
mE_yn \Longleftrightarrow \sigma_m(y) = \sigma_n(y),\; \hbox{ for } m,n\in D.
\]
There are finitely many equivalence relations on the finite set
$D$. Therefore, we can partition $A_D$ into finitely many
$\nu$-measurable sets by putting $y,y'\in A_D$ in the same set
precisely when $E_y = E_{y'}$. These are the sets $A$ of our
countable partition. It is clear that they are disjoint, they cover each $A_D$, and there is countably many of them.
Note further that with each such set $A$ we
can associate an equivalence relation $E$ on $D$ given by $E=E_y$
for each, equivalently any, $y\in A$.

We fix for a moment $A$, $D$ and $E$ as above. Let $D'\subseteq D$
pick precisely one point from each $E$-equivalence class. Then for
each $f\in L^0(\mu, {\mathbb R})$
\begin{equation}\label{E:inttry}
\sum_{n\in D} g_n \cdot (f\circ\sigma_n)\upharpoonright A = \sum_d
\left( (\sum_{n\in d} g_n)\cdot (f\circ
\sigma_{n_d})\upharpoonright A\right),
\end{equation}
where $d$ runs over the set of all $E$-equivalence classes and
$n_d$ is the unique element of $D'\cap d$. On the set $A$, for
$n\in {\mathbb N}$, define
\begin{equation}\notag
k_n=
\begin{cases}
\sum_{m\in d} g_m,&\text{if $n=n_d$ for an $E$-equivalence class
$d$;}\\
0,&\text{ if $n\not\in D'$.}
\end{cases}
\end{equation}
It follows from \eqref{E:inttry} that
\begin{equation}\label{E:rear}
\sum_n g_n \cdot (f\circ\sigma_n)\upharpoonright A = \sum_{n\in D}
g_n\cdot (f\circ\sigma_n)\upharpoonright A =\sum_{n\in D'}
k_n\cdot (f\circ \sigma_{n})\upharpoonright A.
\end{equation}
Note also that by definition of $E$ and $D'$ (so ultimately by
\eqref{E:gzy}) for distinct $m,n\in D'$ and $\nu$-almost every
$y\in A$, we have
\begin{equation}\label{E:dist}
\sigma_m(y) \not= \sigma_n(y).
\end{equation}

We claim that $A$ can be covered by countably many
$\nu$-measurable sets $A_l$, $l\in {\mathbb N}$, for which there
are $\mu$-measurable sets $B^n_l\subseteq X$ with $n\in D'$ so
that for each $l\in {\mathbb N}$ and for distinct $m,n\in D'$ we
have
\begin{equation}\label{E:och}
B^m_l\cap B^n_l =\emptyset\;\hbox{ and }\; \sigma_n(A_l)\subseteq
B^n_l.
\end{equation}
To see this, cover
\[
X^{D'} \setminus \bigcup_{m\not=n, m,n\in D'}\{ (x_i)_{i\in D'}\in
X^{D'}\colon x_m= x_n\}
\]
with sets enumerated by $l\in {\mathbb N}$ and of the
form
\[
\prod_{m\in D'} B^m_l
\]
with $B^n_l\subseteq X$ $\mu$-measurable and with
\[
B^m_l\cap B^n_l = \emptyset
\]
for all $l$ and for distinct $m,n\in D'$. Now let
\[
A_l = \bigcap_{n\in D'}\sigma_n^{-1}(B^n_l).
\]
Condition \eqref{E:dist} ensures that $A$ is covered by the sets
$A_l$. These sets are easily seen to be as required by \eqref{E:och}.

Fix $l\in {\mathbb N}$ and $n_0\in D'$. Since the function
$\chi_{B^{n_0}_l}\in L^0(\mu, {\mathbb R})$ has integer values, by \eqref{E:inva}, we see that
\[
\sum_n g_n \cdot (\chi_{B^{n_0}_l}\circ\sigma_n)
\]
has integer values as well. But by \eqref{E:rear} and by
\eqref{E:och}, we have
\[
\sum_n g_n\cdot (\chi_{B^{n_0}_l}\circ\sigma_n)\upharpoonright A_l
= \sum_{n\in D'} k_n\cdot (\chi_{B^{n_0}_l}\circ
\sigma_{n})\upharpoonright A_l = k_{n_0}\upharpoonright A_l.
\]
Thus, $k_{n_0}\upharpoonright A_l$ has integer values. Since this
is true for each $l\in {\mathbb N}$, we see that
$k_{n_0}\upharpoonright A$ has integer values. This statement
holds for all $n_0\in D'$. Since $k_n\upharpoonright A=0$ for
$n\not\in D'$, we see that all $k_n$ have integer values on $A$.

Since the sets $A$ partition $Y$ up to a set of $\nu$-measure
zero, the functions $k_n$ are defined $\nu$-almost everywhere, they are
integer valued, and, from \eqref{E:rear}, they fulfill 
\[
\sum_n g_n \cdot (f\circ\sigma_n) = \sum_n k_n \cdot (f\circ
\sigma_{n}).
\]
It follows that
\begin{equation}\label{E:prod}
\psi(\exp(2\pi i f)) = \prod_n \exp(2\pi i k_n\cdot
(f\circ\sigma_n)) =\prod_n (\exp(2\pi i f)\circ \sigma_n)^{k_n}.
\end{equation}

Finally, partition $A$ into countably many $\nu$-measurable sets $B$
such that each $k_n$ with $n\in D'$ is constant on each of these
sets. Then by \eqref{E:dist} and \eqref{E:prod} it follows that the
partition of $Y$ into such sets $B$ gives a decomposition of
$L^2(\nu)$ into orthogonal subspaces of the form $L^2(\nu \upharpoonright B)$ that are invariant under
the representation and are such that the representation on each of
them has the following form. There exist a sequence $k_1\leq
\cdots \leq k_n$ of integers and Borel non-singular functions $\sigma_i \colon
B\to X$, $i\leq n$, with
\begin{equation}\label{E:idn}
\sigma_i(y)\not= \sigma_j(y),\hbox{ for }i\not= j \hbox{ and }y\in B
\end{equation}
such that under the representation the operator associated with
$f\in L^0(\mu, {\mathbb T})$ is given by
\[
L^2(\nu\upharpoonright B)\ni h\to (\prod_{i\leq n}(f\circ
\sigma_i)^{k_i})h.
\]

Since the representation is assumed to be non-trivial (as $H_1$ is assumed non-trivial), there is at least one set $B\subseteq Y$ in the decomposition above
such that the sequence $k_1, \dots,
k_n$ associated with it is not constantly equal to $0$. Let $\kappa = (l_1, \dots, l_m)$ be obtained from
$k_1, \dots, k_n$ by deleting all $0$-s, and let $l_1\leq \cdots
\leq l_m$. Note that $\kappa\in S$. Further, let $\tau_1, \dots,
\tau_m$ list the $\sigma_i$-s with $i$-s not corresponding to
$k_i=0$. Note that
\[
\tau= (\tau_1, \dots , \tau_m): B\to X^{|\kappa|}.
\]
Let $\lambda$ be the measure on $X^{|\kappa|}$ obtained from $\mu$
by pushing it forward by $\tau$. Note that non-singularity of each
$\tau_i$ implies that condition~\eqref{E:zer2} holds. Condition~\eqref{E:idn} implies~\eqref{E:zer1}.
Consider the closed space of all elements of $L^2(\nu\upharpoonright B)$
that are constant on the preimages under $\tau$ of
$\lambda$ almost all points in $X^{|\kappa|}$.
(One makes this statement precise, as usual, by disintegrating $\nu \upharpoonright B$ with respect to $\tau$.)
Note that this space is invariant under the
representation, is non-trivial, and the representation
on it is of the form $\sigma(\kappa,\lambda)$.

Thus, as required, we produced a non-trivial subspace $H'$ that is
invariant under the representation and the representation
restricted to $H'$ is unitarily equivalent to $\sigma(\kappa, \lambda)$.

{\em Step 2.} We show here that $\phi$ restricted to $H_0$ is a direct sum of countably many representation of
the form $\sigma(\kappa, \lambda)$ for some $\kappa\in S$ and
$\lambda$ with \eqref{E:zer2}, \eqref{E:zer1}, and also \eqref{E:zer3}.
First, note that Zorn's lemma allows us to pick a maximal
family $\mathcal F$ of mutually orthogonal non-trivial subspaces
$H'$ of $H_0$ such that the representation restricted to each $H'$ is
of the form $\sigma(\kappa, \lambda)$ for $\kappa\in
S$ and $\lambda$ fulfilling \eqref{E:zer2} and \eqref{E:zer1}. By
separability of $H$, $\mathcal F$ is countable. By Step 1,
$\mathcal F$ spans $H_0$. It will suffice, therefore, to represent
each $\sigma(\kappa, \lambda)$ with $\lambda$ fulfilling
\eqref{E:zer2} and \eqref{E:zer1} as a finite direct sum of
representations $\sigma(\kappa, \lambda')$, where $\lambda'$
fulfils \eqref{E:zer2}, \eqref{E:zer1}, and \eqref{E:zer3}.

Fix $\kappa = (k_1, \dots , k_n)$ and $\lambda$ with \eqref{E:zer2} and \eqref{E:zer1}. Let $S_\kappa$
consist of all permutations $\rho$ of $\{ 1, \dots , n\}$ such
that for $1\leq i, j\leq n$
\begin{equation}\label{E:kso}
\rho(i) = j \Longrightarrow k_i=k_j.
\end{equation}
For $\rho\in S_\kappa$, let
\[
X^\rho = \{ (x_1, \dots, x_n)\in X^n\colon x_{\rho(i)}<_X
x_{\rho(j)} \hbox{ for all }i, j\hbox{ with }k_i=k_j\}.
\]
Note that $L^2(\lambda\upharpoonright X^\rho)$ is invariant under
the representation $\sigma(\kappa, \lambda)$. Since $\lambda$
fulfills \eqref{E:zer1}, it follows that $\sigma(\kappa, \lambda)$
is the direct sum of the representations $\sigma(\kappa,
\lambda\upharpoonright X^\rho)$ with $\rho$ varying over
$S_\kappa$. For $\rho\in S_\kappa$, let ${\bar \rho}\colon X^\rho\to
X^{\rm id}$, where $\rm id$ is the identity permutation, be given by
\[
{\bar \rho}(x_1, \dots , x_2) = (x_{\rho(1)}, \dots, x_{\rho(n)}).
\]
From \eqref{E:kso} it is clear that $\sigma(\kappa,
\lambda\upharpoonright X^\rho)$ can be replaced by $\sigma(\kappa,
{\bar \rho}_*(\lambda)\upharpoonright X^{\rm id})$, and so
$\sigma(\kappa, \lambda)$ is the direct sum of the representations
$\sigma(\kappa, {\bar \rho}_*(\lambda)\upharpoonright X^{\rm id})$ with
$\rho$ varying over $S_\kappa$. It is also clear that the measure
${\bar \rho}_*(\lambda)\upharpoonright X^{\rm id}$ fulfills
\eqref{E:zer3}, as well as \eqref{E:zer1} and \eqref{E:zer2}.
Thus, the conclusion follows.

{\em Step 3.} We now assume that the representation $\phi$ when
restricted to $H_0$ is the direct sum as described in Step 2. We
show how to modify this direct sum so that it fulfills
\eqref{E:absas} as in the conclusion of the theorem. Fix
$\kappa\in S$. Let $\lambda^j$, $i< m$, list all non-zero measures
on $X^{|\kappa|}$ appearing in $\sigma(\kappa, \lambda^j)$ in the
direct sum given by Step 2. Here $m\in {\mathbb N}\cup\{ \infty\}$. We can,
and we do, assume that each $\lambda^j$ is a probability measure. We will
use the following general and easy observation. Assume we have
finite Borel measures
\[
\nu_{i-1}\ll \cdots \ll \nu_2\ll \nu_1
\]
on a standard Borel space and another finite Borel measure $\mu$
on the same space. Then $\mu = \mu_1+\cdots +\mu_i$, where
\begin{equation}\notag
\begin{split}
&\mu_j\ll \nu_{j-1}, \hbox{ for }2\leq j\leq i;\\
&\mu_j\perp \nu_j, \hbox{ for }1\leq j\leq i-1;\\
&\mu_j\perp \mu_{j'}, \hbox{ for }1\leq j, j'\leq i, j\not=j'.
\end{split}
\end{equation}
Using this observation, by induction on $i$, we find finite Borel measures $\lambda^j_j$
with $j\leq i$ so that the following conditions hold
\begin{enumerate}
\item[(a)] $\lambda^i= \lambda^i_1 + \cdots + \lambda^i_i$;

\item[(b)] $\lambda^i_i\ll\cdots \ll \lambda_2^2 + \cdots +
\lambda_2^i\ll \lambda_1^1 + \lambda_1^2+ \cdots +\lambda_1^i$;

\item[(c)] $\lambda_j^i\perp (\lambda^j_j+\cdots +
\lambda^{i-1}_j),\; \hbox{ for }j<i$;

\item[(d)] $\lambda^i_j\perp \lambda^i_{j'},\;\hbox{ for }
j,j'\leq i,\, j\not= j'$.
\end{enumerate}
Note that condition (b) for $i-1$ is used as an inductive
assumption and is maintained in the induction by the first
condition in the general observation above.

Let now $\lambda_\kappa^j$ for $j<m$ be the measure
\[
\lambda^j_j + 2^{-1}\lambda^{j+1}_j + 2^{-2}\lambda^{j+2}_j
+\cdots.
\]
This is a finite measure by (a) since each $\lambda^i$ was assumed
to be a probability measure. Set also $\lambda^j_\kappa=0$ for
$j\in {\mathbb N}$ and $j\geq m$. By conditions (a) and (d), the
direct sum of the representations $\sigma(\kappa, \lambda^i)$ for
$i\in {\mathbb N}$ is the direct sum of
the representations $\sigma(\kappa, \lambda^i_j)$ for $j\leq i$,
$i\in {\mathbb N}$. By condition (c) and the definition of
$\lambda^j_\kappa$, this latter direct sum is also
the direct sum of the representations $\sigma(\kappa,
\lambda^j_\kappa)$ for $j\in {\mathbb N}$. Condition (b) ensures
that $\lambda^{j'}_\kappa\ll \lambda^j_\kappa$ for $j'>j$, that
is, \eqref{E:absas} holds. Thus, the measures $\lambda_\kappa^j$,
$\kappa\in S$, $j\in {\mathbb N}$, are as required.

\smallskip

\noindent {\bf Proof of uniqueness of the presentation.}
For subsets $P,\, Q$ of $X$, we write
\[
P<_X Q
\]
if $x<_X y$ for all $x\in P$ and $y\in Q$.

We will use the following elementary observation, whose
justification we leave to the reader.

\noindent Observation. {\em Let $\kappa= (k_1, \dots , k_m)$ and $\kappa' = (l_1,
\dots , l_n)$ be in $S$. Let $q\in {\mathbb N}$ and let
\[
u\colon \{ 1, \dots , m\} \to \{ 1, \dots, q\}\hbox{ and } v\colon
\{ 1, \dots , n\} \to \{ 1, \dots, q\}
\]
be injective. Assume that for all $z_1, \dots, z_q\in {\mathbb T}$,
we have
\[
z_{u(1)}^{k_1}\cdots z_{u(m)}^{k_m} =  z_{v(1)}^{l_1}\cdots
z_{v(n)}^{l_n}.
\]
Then $m=n$ and for each $r\in \{ 1, \dots, q\}$
\[
\{ u(i)\colon k_i = r\} = \{ v(i)\colon l_i = r\};
\]
therefore, since $\kappa, \kappa'\in S$, $\kappa=\kappa'$. It follows that, if additionally for all $i<j\leq m$
\[
k_i=k_j\Rightarrow u(i)<u(j)\;\hbox{ and }\; l_i=l_j\Rightarrow v(i)<v(j),
\]
then $u=v$. }

We will also consider $X$ equipped with the order topology induced by $<_X$, which
is assumed to be second countable and compact.

Assume we are given a representation $\phi$ of $L^0(\mu, {\mathbb
T})$ on a separable Hilbert space $H$. Assume that we have two
presentations of the restriction of $\phi$ to $H_0$ given by
$\lambda^j_\kappa$ and by $(\lambda')^j_\kappa$, for $i\in
{\mathbb N}$ and $\kappa\in S$. Assume further towards a contradiction
that these two presentations do not fulfill the uniqueness
criterion from the theorem. Thus, there exist $\kappa$, $j$, and a
Borel set $K\subseteq X^{|\kappa|}$ whose measure is positive with
respect to one of the measures $\lambda^j_\kappa$,
$(\lambda')^j_\kappa$ and is zero with respect to the other. Fix
such a $\kappa$ and such a $j$. They will be called $\kappa_0$ and
$j_0$, respectively. Let $\kappa_0$ be equal to $(k_1, k_2, \dots,
k_n)$; in particular, $|\kappa_0|=n$. Without loss of generality,
we can assume that for all $j<j_0$, we have
$\lambda^j_{\kappa_0}\sim (\lambda')^j_{\kappa_0}$, that
\begin{equation}\label{E:Kineq}
\lambda^{j_0}_{\kappa_0}(K)>0\;\hbox{ and }\;
(\lambda')^{j_0}_{\kappa_0}(K)=0,
\end{equation}
and, by using \eqref{E:zer1} and \eqref{E:zer3}, that
\begin{equation}\label{E:Kin}
K\subseteq \{ x\in X^n\colon x_1<_X x_2<_X\cdots <_X x_n\}.
\end{equation}
We can also assume, by shrinking $K$ if necessary, that $K$ is compact
in the product topology on $X^n$.

A partition $\mathcal P$ of $X$ into Borel sets will be called
{\em admissible} if $P<_X Q$ or $Q<_X P$ for distinct $P,Q\in
{\mathcal P}$ and
\[
K\subseteq \bigcup_u \prod_{i=1}^nu(i),
\]
where $u$ varies over the set of all functions $u\colon \{ 1, \dots, n\} \to {\mathcal P}$
with
\begin{equation}\label{E:vv}
u(i_1)<_X u(i_2), \hbox{ for }1\leq i_1<i_2\leq n.
\end{equation}
Using \eqref{E:Kin} and compactness of $K$,
we see that
there exists an admissible partition. Note that a partition finer
than an admissible one is also admissible. Further, for an admissible partition $\mathcal P$ and $j\in {\mathbb N}$, put
\[
U^j_{\mathcal P} = \{ u\colon \{ 1, \dots, n\} \to {\mathcal P}\colon u \hbox{ fulfills \eqref{E:vv} and } \lambda^j_{\kappa_0}(K\cap \prod_{i=1}^n
u(i))>0\} .
\]
With a sequence $z_{\mathcal P} = (z_P\in {\mathbb T}\colon P\in
{\mathcal P})$, we associate the unitary operator
\begin{equation}\label{E:bigf}
A_{{\mathcal P}, z_{\mathcal P}} = \phi(\sum_{P\in {\mathcal P}}
z_P\chi_P).
\end{equation}

Define $H_K$ to be the closure of the set of all $h\in H$ with the
following property. For every admissible partition ${\mathcal P}$
of $X$, one can represent $h$ as
\begin{equation}\label{E:bigs}
h = \sum_{u\in U_{\mathcal P}}h_u,
\end{equation}
where, for every $z_{\mathcal P}$, $h_u$ is an eigenvector of $A_{{\mathcal P}, z_{\mathcal
P}}$ with eigenvalue
\begin{equation}\label{E:bige}
\prod_{i=1}^n z_{u(i)}^{k_i}.
\end{equation}

We claim that when $\phi$ is viewed as the
direct sum of $\sigma(\kappa, \lambda^j_\kappa)$ for $j\in {\mathbb N}$
and $\kappa\in S$, then $H_K$ is the direct sum of
$L^2(\lambda^j_{\kappa_0}\upharpoonright K)$ over $j\in {\mathbb
N}$, and when $\phi$ is viewed as the
direct sum of $\sigma(\kappa, (\lambda')^j_\kappa)$, then $H_K$ is the direct sum of
$L^2((\lambda')^j_{\kappa_0}\upharpoonright K)$ over $j\in
{\mathbb N}$. We write out a proof only for $\lambda^j_\kappa$. In it, we identify $H$ with the direct
sum of $L^2(\lambda^j_\kappa)$ over $j\in {\mathbb N}$ and $\kappa\in S$.
That the direct sum of $L^2(\lambda^j_{\kappa_0}\upharpoonright K)$, $j\in {\mathbb N}$, $\kappa\in S$, is included in $H_K$ is not difficult to check
and we leave it to the reader. We show the other inclusion. Assume
for contradiction that it does not hold. Then we have an element
$h$ of $H_K$ whose projection on $L^2(\lambda^j_{\kappa})$ is
non-zero for some $j$ and some $\kappa\not= \kappa_0$ or whose
projection on $L^2(\lambda^j_{\kappa_0})$, for some $j$, has
support not included in $K$. Let $A\subseteq X^{|\kappa|}$ be the
support of this projection. Set $m=|\kappa|$ and $\kappa = (l_1,
\dots , l_m)$. By condition \eqref{E:zer1} and by compactness of $K$, there is an admissible
partition $\mathcal P$ of $X$ and an injective function $v\colon
\{ 1, \dots, m\}\to {\mathcal P}$ such that
\begin{equation}\label{E:possu}
\lambda^j_\kappa\big( A\cap (v(1)\times \cdots \times v(m))\big) >0,
\end{equation}
and either $\kappa\not= \kappa_0$ or ($\kappa=\kappa_0$ and  $v\not\in U^j_{\mathcal P}$). Note that, by \eqref{E:zer3} and by
\eqref{E:possu}, $v$ satisfies for all $1\leq i_1<i_2\leq m$
\begin{equation}\label{E:ratu}
l_{i_1}=l_{i_2}\Longrightarrow v(i_1)<_X v(i_2).
\end{equation}

Now if $h$ is represented as a sum as in \eqref{E:bigs} for the
$\mathcal P$ found above, then there exists an $h_u$, for some $u\in U_{\mathcal P}$, whose projection on
$L^2(\lambda^{j}_{\kappa})$ has support intersecting $v(1)\times
\cdots \times v(m)$ on a set of $\lambda^j_\kappa$ positive measure. Now,
$h_u$ is an eigenvector of \eqref{E:bigf} for every $z_{\mathcal P}$. Its eigenvalue for a given $z_{\mathcal P}$ must be
equal to
\[
\prod_{i=1}^{m} z_{v(i)}^{l_i}
\]
since every value of a function from $L^2(\lambda^{j}_{\kappa})$
attained on $v(1)\times \cdots \times v(m)$ is multiplied by that
number when the function is acted on by \eqref{E:bigf}. On the other hand, this
eigenvalue is also equal to \eqref{E:bige} for the $u$ found above. Now, using the observation from the beginning
of the proof of uniqueness and using \eqref{E:ratu} and \eqref{E:vv} for $u$, we see that $\kappa =
\kappa_0$ and $v=u$, so $v\in U^j_{\mathcal P}$, contradiction.
Thus, $H_K$ is the direct sum of
$L^2(\lambda^j_{\kappa_0}\upharpoonright K)$ over $i\in {\mathbb
N}$. Similarly, we get that it is the direct sum of
$L^2((\lambda')^j_{\kappa_0}\upharpoonright K)$ over $j\in
{\mathbb N}$.

Using the presentation of $H_K$ as a direct sum with respect to $(\lambda')^j_\kappa$,
$\kappa\in S$, $j\in {\mathbb N}$, we show that there are $j_0-1$
vectors such that $H_K$ is the smallest closed subspace containing
these vectors and invariant under $L^0(\mu, {\mathbb T})$. Take a
copy of $\chi_K$ in each $L^2((\lambda')^j_{\kappa_0})$ for
$j<j_0$. Note that vectors obtained from each of the $j_0-1$ copies of $\chi_K$ by
acting on them by elements of $L^0(\mu, {\mathbb T})$ separate points of $K$. So the
smallest closed subspace $H'$ containing all these vectors contains
each $L^2((\lambda')^j_{\kappa_0}\upharpoonright K)$ for $j<j_0$.
Since $(\lambda')^j_{\kappa_0}(K)=0$ for $j\geq j_0$, we see that $H'$
contains the direct sum of $L^2((\lambda')^j_{\kappa_0}\upharpoonright K)$ over all $j$, that is,
by what was proved above, it is equal to $H_K$.

On the other hand, using the presentation of $H_K$ as a direct sum with respect to
$\lambda^j_\kappa$, $\kappa\in S$, $j\in {\mathbb N}$, we show
that given $j_0-1$ vectors $f_j$, $j<j_0$, in $H_K$ the closed subspace $H'$ spanned by
all the vectors obtained from $f_j$, $j<j_0$, by acting on them by $L^0(\mu, {\mathbb
T})$ is a proper subspace of $H_K$. First note that since $H_K$ is the direct sum of
$L^2(\lambda^j_{\kappa_0}\upharpoonright K)$, for $j\in {\mathbb N}$,
with each element $h\in H_K$ we can associate a sequence $h^j$, $j\in {\mathbb N}$,
with each $h^j$ being a $\lambda^j_{\kappa_0}$ class of
a function from $K$ to $\mathbb C$. Since $\lambda^{j_0}_{\kappa_0}\ll \lambda^j_{\kappa_0}$ for $j\leq j_0$, we
see that $h^j$, for $j\leq j_0$, determines a single $\lambda^{j_0}_{\kappa_0}$ function class, which we again denote by $h^j$.
Then we define $\overline{h}\in L^2(\lambda^{j_0}_{\kappa_0}\upharpoonright K, {\mathbb C}^{j_0})$ by letting for $\lambda^{j_0}_{\kappa_0}$
almost every $x\in K$
\[
\overline{h}(x) = (h^1(x), \dots, h^{j_0}(x)) \in {\mathbb C}^{j_0}.
\]
Now, for $\lambda^{j_0}_{\kappa_0}$ almost every $x\in K$, let $V_x$ be the linear
subspace of ${\mathbb C}^{j_0}$ spanned by $\overline{f_1}, \dots , \overline{f_{j_0-1}}$. Note that
the function $K\ni x\to V_x$ is $\lambda^{j_0}_{\kappa_0}$ measurable and that the dimension of $V_x$ does
not exceed $j_0-1$. Observe that for $g\in L^0(\mu, {\mathbb T})$ and $h\in H_K$,
\[
\overline{\phi(g)(h)}(x) = z_x\bar{h}
\]
for some $z_x\in{\mathbb C}$ for $\lambda^{j_0}_{\kappa_0}$ almost every $x\in K$. It
follows that for each $h\in H'$, we have
\begin{equation}\label{E:ins}
\overline{h}(x)\in V_x,\hbox{ for }\lambda^{j_0}_{\kappa_0}\hbox{ almost every }x\in K.
\end{equation}
Since the dimension of $V_x$ is less than that of ${\mathbb C}^{j_0}$, we have that ${\mathbb C}^{j_0}\setminus V_x$
is non-empty for $\lambda^{j_0}_{\kappa_0}$ almost every $x\in K$. Thus, by the Jankov--von Neumann selection
theorem, see \cite[Theorem 18.1]{Ke}, there is a bounded $\lambda^{j_0}_{\kappa_0}$ measurable function
$F\colon K\to {\mathbb C}^{j_0}$ such that
\begin{equation}\label{E:rg}
F(x)\not\in V_x, \hbox{ for }\lambda^{j_0}_{\kappa_0}\hbox{ almost every }x\in K.
\end{equation}
We easily see that there is $f\in H_K$ with
\[
\overline{f} = F.
\]
It follows from \eqref{E:ins}, \eqref{E:rg}, and the fact that $\lambda^{j_0}_{\kappa_0}(K)>0$ that $f$ is not in $H'$.
Thus, $H'$ is a proper subspace of $H_K$.
This conclusion yields
a contradiction and completes the proof of uniqueness of the
presentation.

\section{Unitary representations of $C(M, {\mathbb T})$}

In this section, $M$ is a second countable,  compact, zero-dimensional space.
We also fix a linear order $<_M$ on $M$ such that the order topology induced by it
is compact and second countable and has the same Borel sets as the original topology on $M$.
In fact, since $M$ is zero-dimensional, it can be viewed as a subset of $\{ 0,1\}^{\mathbb N}$,
see \cite[Theorem 7.8]{Ke}, and the order $<_M$ can be defined to be the pull-back of the lexicographic order
on $\{ 0,1\}^{\mathbb N}$. Then the order topology induced by $<_M$ is equal to the original topology on $M$.

\subsection{Description of representations}\label{Su:des}
The description here is essentially the one from
Subsection~\ref{Su:des2} except that, obviously, we do not have
a condition analogous to \eqref{E:zer2}. We keep the piece of notation $S$
standing for the set of all $\kappa = (k_1, \dots , k_n)$ of
elements of ${\mathbb Z}\setminus \{ 0\}$ with $k_1\leq k_2\leq
\cdots \leq k_n$.

Given $\kappa = (k_1, \dots , k_n)\in S$ and a finite Borel
measure $\lambda$ on $M^{n}$, we consider the representation of
$C(M,{\mathbb T})$ on $L^2(\lambda, {\mathbb C})$ given by:
\[
C(M, {\mathbb T})\ni f\to U_f\in {\mathcal U}(L^2(\lambda, {\mathbb C})),
\]
where for $h\in L^2(\lambda, {\mathbb C})$
\[
U_f(h) = (\prod_{i\leq n} (f\circ \pi_i)^{k_i}) h,
\]
where, for $i\leq n$, $\pi_i\colon M^n\to M$ is the projection on
the $i$-th coordinate. We denote this representation again by
$\sigma(\kappa, \lambda)$.

For $\kappa\in S$ with $n= |\kappa|$ and a finite Borel measure
$\lambda$ on $M^n$, we will consider the following two conditions:
for all $1\leq i<j\leq n$
\begin{equation}\label{E:zerc}
\lambda(\{ (x_1, \dots , x_n)\in M^n\colon x_i=x_j\}) =0 \tag{B1}
\end{equation}
and, for $1\leq i<j\leq n$ with $k_i=k_j$,
\begin{equation}\label{E:zerc2}
\lambda(\{ (x_1, \dots , x_n)\in M^n\colon x_j <_M x_i\}) =0. \tag{B2}
\end{equation}

\subsection{Statement of the theorem}\label{Su:res}

\begin{theorem}\label{T:CCmain}
Let $M$ be a compact, second countable, zero-dimensional space.
Let $\phi$ be a continuous unitary representation of
$C(M,{\mathbb T})$ on a separable complex Hilbert space $H$.
Consider $H_0$, the orthogonal complement of
\[
\{ v\in H\colon \forall f\in C(M, {\mathbb T})\; \phi(f)(v) =
v\}.
\]
\noindent {\rm Existence.} For each $\kappa\in S$, there exist
finite Borel measures $\lambda_\kappa^i$, $i\in {\mathbb N}$, on
$M^{|\kappa|}$ with properties \eqref{E:zerc}, \eqref{E:zerc2},
and with
\begin{equation}\label{E:itsl}
\lambda_\kappa^j\ll \lambda_\kappa^i, \hbox{ for }i<j, \tag{B3}
\end{equation}
such that the representation $\phi$ restricted to $H_0$ is
the direct sum of the representations $\sigma(\kappa,
\lambda_\kappa^i)$ with $\kappa\in S$ and $i\in {\mathbb N}$.

\noindent {\rm Uniqueness.} If the restriction of $\phi$ to $H_0$
is presented as the direct sum of $\sigma(\kappa,
(\lambda')^i_{\kappa})$, for $\kappa\in S$, $i\in {\mathbb N}$,
with \eqref{E:zerc}, \eqref{E:zerc2}, and \eqref{E:itsl}, then
$\lambda^i_\kappa$ and $(\lambda')^i_\kappa$ are absolutely
continuous with respect to each other for all $\kappa\in S$ and
$i\in {\mathbb N}$.
\end{theorem}

It is not difficult to see that the theorem above fails without the assumption
of zero-dimensionality, for example, it fails for $M=[0,1]$.

We have a corollary similar to Corollary~\ref{C:measl}. Recall
that given a finite Borel measure $\nu$ on a compact second
countable space $M$, mapping a function from $C(M, {\mathbb T})$
to its equivalence class in $L^0(\nu, {\mathbb T})$ induces a continuous
homomorphism from the first group to the latter.

\begin{corollary}\label{C:measc}
Let $M$ be a compact second countable zero-dimensional space. Let $\phi$ be a
continuous unitary representation of $C(M, {\mathbb T})$. There
exists a finite Borel measure $\nu$ on $M$ such that
\begin{enumerate}
\item[(i)] $\phi$ is the composition of the natural homomorphism
from $C(M, {\mathbb T})$ to $L^0(\nu, {\mathbb T})$ and a
continuous unitary representation of $L^0(\nu, {\mathbb T})$;

\item[(ii)] if $\nu'$ is a finite Borel measure on $M$ such that
$\phi$ is the composition of the natural homomorphism from $C(M,
{\mathbb T})$ to $L^0(\nu', {\mathbb T})$ and a continuous unitary
representation of $L^0(\nu', {\mathbb T})$, then $\nu\ll \nu'$.
\end{enumerate}
\end{corollary}

Before proving the corollary we recall a simple lemma that will also
be used in the proof of Theorem~\ref{T:CCmain} and whose proof we leave to the reader.

\begin{lemma}\label{L:emgf}
Let $\nu$ be a finite Borel measure on a compact second countable
space $M$. The image of the natural homomorphism from $C(M,
{\mathbb T})$ to $L^0(\nu, {\mathbb T})$ is dense in $L^0(\nu,
{\mathbb T})$.
\end{lemma}

\begin{proof}[Proof of Corollary~\ref{C:measc}]
This corollary follows from Theorem~\ref{T:CCmain} in a manner
essentially identical to the argument showing
Corollary~\ref{C:measl}. The only addition is the following
remark. After the measure $\nu$ is produced, we get that the
representation $\phi$ remains continuous when $C(M, {\mathbb
T})$ is taken with the $L^0$ topology with respect to $\nu$. At
this point we use Lemma~\ref{L:emgf} to extend $\phi$ to a
continuous unitary representation of $L^0(\nu, {\mathbb T})$.
After that the proof again follows the route of the proof of
Corollary~\ref{C:measl}.
\end{proof}

We will give the proof of Theorem~\ref{T:CCmain} in
Subsection~\ref{Su:pr}. In Subsection~\ref{Su:moba}, we collect
some more results needed for the argument.

\subsection{More background results}\label{Su:moba}

As explained below the following result is a combination of the
factorization theorems from \cite{N}, \cite{Me} and \cite{P}.

\begin{proposition}\label{P2}
Let $M$ be a compact, second countable space, and let $\mu$ be a Borel
probability measure on a standard Borel space. Each continuous
operator from $C(M, {\mathbb R})$ to $L^0(\mu, {\mathbb R})$
factors through a real Hilbert space.
\end{proposition}

By \cite{N} each operator $C(M, {\mathbb R})\to L^0(\mu, {\mathbb
R})$ factors through $L^p(\mu, {\mathbb R})$ for each $0\leq p<1$;
by \cite[Th{\'e}or{\`e}me III,2]{Me}, for each $0<p<1$, each operator
$C(M, {\mathbb R})\to L^p(\mu, {\mathbb R})$ factors through
$L^q(\mu, {\mathbb R})$ for some $1<q\leq 2$; by \cite[Corollary
4.4]{P}, for each $1\leq q\leq 2$, each operator $C(M, {\mathbb
R})\to L^q(\mu, {\mathbb R})$ factors through a Hilbert space. A
simple concatenation of these three results gives the proposition
above. It is possible that the proposition can be
obtained by the methods of \cite[Corollaire 34, Th{\'e}or{\`e}me 93]{M},
where the same result is proved with $C(M, {\mathbb R})$ replaced
by $L^\infty$.

We will need the following result that is a combination of
theorems of Grothendieck and Pietsch, see \cite[Theorem 5.4(b),
Corollary 1.5]{P}.

\begin{proposition}\label{P:GP}
Let $M$ be a compact, second countable space and let $H$ be a separable
real Hilbert space. Let $T\colon C(M, {\mathbb R})\to H$ be a
continuous linear operator. There exists a Borel probability
measure $\nu$ on $M$ such that $T$ remains continuous if we take
$C(M,{\mathbb R})$ with the $L^2$ topology with respect to $\nu$.
\end{proposition}

\subsection{Proof of Theorem~\ref{T:CCmain}}\label{Su:pr}

As before, when we say ``representation" we mean ``strongly continuous unitary
representation."

Assume we have a representation $\phi$ of
$C(M,{\mathbb T})$ on a separable complex Hilbert space.
As in the start of the proof of Theorem~\ref{T:LLmain}, we use the
spectral theorem, Proposition~\ref{P:sp}, to see that the unitary
representation $\phi$ of $C(M, {\mathbb T})$ is
the form
\[
C(M, {\mathbb T})\ni f\to N_f\in {\mathcal U}(L^2(\nu, {\mathbb C}))
\]
where
\[
N_f(h) = \psi(f)\cdot h
\]
for some continuous homomorphism
\begin{equation}\label{E:homc}
\psi\colon C(M, {\mathbb T})\to L^0(\nu, {\mathbb T}).
\end{equation}

Consider now the exponential map
\begin{equation}\notag
C(M, {\mathbb R})\to C(M, {\mathbb T});\; f\to \exp(2\pi i f).
\end{equation}
Note that $C(M, {\mathbb T})$ is
connected since $M$ is zero-dimensional. Since the image of $C(M, {\mathbb R})$ under
the exponential map contains an open ball
around the identity in $C(M, {\mathbb T})$, by connectedness of $C(M, {\mathbb T})$, it is actually equal to $C(M, {\mathbb
T})$. Thus, we have a surjective continuous homomorphism
\begin{equation}\label{E:expc}
C(M, {\mathbb R})\to C(M, {\mathbb T});\; f\to \exp(2\pi i f).
\end{equation}
By precomposing $\psi$ with the exponential homomorphism
\eqref{E:expc}, we obtain a continuous homomorphism $\psi' \colon
C(M,{\mathbb R})\to L^0(\nu, {\mathbb T})$. By
Proposition~\ref{P1}, there exists a continuous linear operator
\[
\theta\colon C(M,{\mathbb R})\to L^0(\nu, {\mathbb R})
\]
such that
\[
\psi' = \exp(2\pi i\theta).
\]
By Proposition~\ref{P2}, $\theta$ factors through a Hilbert space,
that is, there is a continuous linear operator $\theta'\colon C(M,
{\mathbb R}) \to H$, where $H$ is a separable real Hilbert space,
such that $\theta$ is the composition of $\theta'$ and a
continuous linear operator $H\to L^0(\nu, {\mathbb R})$. We can
assume that $H$ is separable since $C(M, {\mathbb R})$ is. Now by
Proposition~\ref{P:GP}, there exists a probability Borel measure
$\mu$ on $M$ such that $\theta'$ remains continuous if $C(M,
{\mathbb R})$ is taken with the $L^2$ topology with respect to
$\mu$.

We make two observations. First, since $M$ is zero-dimensional, for each $f\in C(M, {\mathbb T})$, there exists $g\in C(M, {\mathbb R})$
with $|g|< 2/3$ and $f = \exp(2\pi i g)$. 
Our second observation consist of noticing that the $L^2$
and the $L^0$ topologies with respect to $\mu$ are identical on the subset of $C(M, {\mathbb R})$ consisting
of all functions bounded by $2/3$. Now, given a sequence $(f_n)$ in $C(M, {\mathbb T})$ convergent to $1$ in the $L^0$ topology with respect to
$\mu$, pick a sequence $(g_n)$
in $C(M, {\mathbb R})$ such that
$|g_n|<2/3$ and $f_n = \exp(2\pi i g_n)$ for each $n$. It follows that $(g_n)$ converges to $0$ in the $L^0$ topology with respect to $\mu$. Thus,
we get that $(g_n)$ converges to $0$ in the $L^2$ topology, which implies that the sequence $(\psi'(g_n))$ converges to $1$ in
$L^0(\nu, {\mathbb T})$. Since, by definition of $\psi'$, $\psi'(g_n) = \psi(f_n)$, we see that $(\psi(f_n))$ converges to $1$ in
$L^0(\nu, {\mathbb T})$.
It follows that the homomorphism $\psi$ given by \eqref{E:homc}
remains continuous if $C(M, {\mathbb T})$ is taken with the $L^0$
topology with respect to $\mu$.

By Lemma~\ref{L:emgf} we can extend $\psi$ to a continuous homomorphism
\[
\widetilde{\psi}\colon L^0(\mu, {\mathbb T})\to L^0(\nu, {\mathbb
T}).
\]
We apply now Theorem~\ref{T:LLmain} to the representation induced
by this homomorphism and the existence part of
Theorem~\ref{T:CCmain} follows immediately.

The uniqueness part is a consequence of the uniqueness part of
Theorem~\ref{T:LLmain}. Assume we have two presentations as in
Theorem~\ref{T:CCmain} given by $(\lambda^i_\kappa)$ and
$((\lambda')^i_\kappa)$ for $\kappa\in S$ and $i\in {\mathbb N}$
of a single representation $\phi$ of $C(M, {\mathbb T})$
restricted to $H_0$. We define a finite Borel measure $\mu$ on $M$
as follows. For each $\lambda^i_\kappa$ consider all finitely many
measures on $M$ obtained from $\lambda^i_\kappa$ by pushing it
forward by all the projections from $M^{|\kappa|}$ to $M$. Collect
all such push-forwards of $\lambda^i_\kappa$ for all $i$ and
$\kappa$ and form a weighted sum of this countable collection of
finite Borel measures on $M$ obtaining $\mu$. In the same manner
define $\mu'$ from $(\lambda')^i_\kappa$ for $i\in {\mathbb N}$
and $\kappa\in S$. Now it is clear, using either one of the two
presentations, that the representation $\phi$ remains continuous
when we consider $C(M, {\mathbb T})$ with the $L^0$ topology
with respect to the measure $\mu+\mu'$. Using density of $C(M,
{\mathbb T})$ in $L^0(\mu+\mu', {\mathbb T})$, we see
that $\phi$ extends to a unitary representation $\widetilde\phi$
of $L^0(\mu+\mu', {\mathbb T})$. Now both presentation, the one
given by $(\lambda^i_\kappa)$ and the one given by
$((\lambda')^i_\kappa)$, are presentations of $\widetilde\phi$
restricted to $H_0$ as in Theorem~\ref{T:LLmain}. By the
uniqueness part of that theorem, we get $\lambda^i_\kappa\sim
(\lambda')^i_\kappa$ for all $i$ and $\kappa$ as required.

\end{document}